\documentclass[11pt]{amsart}

\usepackage[utf8]{inputenc}
\usepackage{style_scattering}
\usepackage{bm}

\title[Stable blowup under random perturbations]{Stable blowup for the focusing energy critical nonlinear wave equation under random perturbations}
\author{Bjoern Bringmann}

\newcommand{\Addresses}{{% additional braces for segregating \footnotesize
  \bigskip\bigskip \small

  Bjoern Bringmann, \textsc{University of California, Los Angeles, Department of Mathematics, 520 Portola Plaza, Los Angeles, CA 90095}\\\nopagebreak
  Email address: \texttt{bringmann@math.ucla.edu}

}}

\usepackage{enumitem}

\begin{document}
\maketitle
\begin{abstract}
We consider the radial focusing energy critical nonlinear wave equation in three spatial dimensions. We establish the stability of the ODE-blowup under random perturbations below the energy space. The argument relies on probabilistic Strichartz estimates in similarity coordinates.
\end{abstract}
\setcounter{tocdepth}{1}
\tableofcontents
\section{Introduction}

We consider the focusing quintic nonlinear wave equation in three dimensions
\begin{equation}\label{nlw}
\begin{cases} 
- \partial_t^2 u(t,x) + \Delta u(t,x)= - u(t,x)^5 \qquad (t,x) \in \R \times \R^3, \\
u(0,x)=u_1 \in H^s(\R^3), \qquad \partial_t u(0,x) =u_2 \in H^{s-1}(\R^3). 
\end{cases}
\end{equation}

Here, \( s>0 \) is the regularity and \( H^{s}(\R^3) \) denotes the inhomogeneous Sobolev space with \( s\)-derivatives. The flow of the nonlinear wave equation \eqref{nlw} conserves the energy 
\begin{equation}\label{energy}
E[u]=E[u](t) := \int_{\R^3}  \frac{(\partial_t u(t,x))^2}{2} + \frac{|\nabla_x u(t,x)|^2}{2} - \frac{u(t,x)^6}{6} \dx 
\end{equation}
The nonlinear wave equation \eqref{nlw} admits the scaling symmetry \( u(t,x) \mapsto u_\lambda(t,x):= \lambda^{\frac{1}{2}} u(\lambda t, \lambda x) \). Since the scaling symmetry also preserves the energy of the solution, the equation \eqref{nlw} is energy critical.\\ 

The focusing nonlinear wave equation \eqref{nlw} displays a variety of different dynamical behaviors, such as scattering, solitons, or finite-time blowup. While we will also comment on scattering and solitons below, the main focus of this paper lies on the formation of finite-time blowup. In certain situations, blowup may simply indicate a breakdown of the underlying model. In several applications, however, blowup or singularity formation describes real physical phenomena. It is related to the self-focusing effect in nonlinear optics  \cite{BLS09} and the formation of black holes through gravitational collapse \cite{HE73}. Due to these physical phenomena, we are not only interested in the existence of blowup, but also care about the blowup profile and its stability properties. 
We now focus on the so-called \emph{ODE-blowup}, which is given by
\begin{equation}\label{intro:ode_blowup}
u^{(T)}(t,x) := \kappa (T-t)^{-\frac{1}{2}}, \qquad \text{where } T>0 ~ \text{and}~ \kappa := \Big( \frac{3}{4} \Big)^{\frac{1}{4}}. 
\end{equation}
While \eqref{intro:ode_blowup} does not exhibit any spatial decay and hence does not belong to any Sobolev space, we can use finite speed of propagation to localize \( u^{(T)} \) to a lightcone. There is a large amount of literature on stable blowup in nonlinear wave equations and we refer the interested reader to \cite{DZ18,DS12,DS14,DS16,KST09,MZ03,MZ05,MZ07}. The techniques and results used in this paper are closely related to previous work by Donninger \cite{Donninger17}. He proved that the one-parameter family \eqref{intro:ode_blowup} is stable under small \emph{radial} perturbations in the energy space \( H^1(\R^3) \times L^2(\R^3) \). In light of the breakdown of (deterministic) local well-posedness below the energy space (see e.g. \cite{CCT03}), we expect this to be the optimal regularity. Without the radial symmetry assumption, Donninger and Schörkhuber \cite{DS16} proved the stability of \eqref{intro:ode_blowup} under small perturbations in \( H^2(\R^3) \times H^1(\R^3)\). \\

Let us now briefly discuss scattering and solitons. This behavior is intimately tied to the ground state $W$, which is given by the explicit formula  \( W(x):= (1+|x|^2/3)^{-\frac{1}{2}} \). Up to scaling and a sign change, it is the unique radial solution in \( \dot{H}^1(\R^3) \) of the elliptic equation
\begin{equation*}
\Delta W(x) = - W(x)^5. 
\end{equation*}
In a seminal paper \cite{KM08},  Kenig and Merle proved that any initial data satisfying 
\begin{equation*}
E[u_1,u_2] < E[W,0] \qquad \text{and} \qquad \| \nabla u_1 \|_{L^2} < \| \nabla W\|_{L^2}
\end{equation*}
leads to a global solution which scatters as \( t \rightarrow \pm \infty \). In contrast, the ground state \( W \) itself leads to a stationary solution and hence does not scatter. By applying the scaling symmetry or Lorentz transformations to the ground state \( W \), one can generate a whole family of traveling wave solutions, which are also called solitons. As for the ODE-blowup, we are interested in the stability of the soliton evolution under small perturbations. In \cite{Beceanu14,KNS15}, it is proven (in different topologies) that the evolution of the solitons has a single unstable direction and is stable with respect to perturbations in a Lipschitz manifold of codimension one.  We also mention important progress on the soliton resolution conjecture   by Duyckaerts, Kenig, and Merle \cite{DKM17,DKM11,DKM12,DKM13}. \\

Throughout the last decade, there has been growing interest in a probabilistic approach to nonlinear dispersive equations. In contrast to a deterministic well-posedness or stability theory, which has to apply to every initial data in the relevant function space, a probabilistic approach is only concerned with random initial data. In physical applications, this randomness may be a result of microscopic fluctuations in temperatures or densities. As a result, the random initial data of interest only belongs to low-regularity spaces, which often lie below the (deterministic) regularity threshold. This approach first appeared in seminal work of Bourgain \cite{Bourgain94,Bourgain96} and Burq-Tzvetkov \cite{BT08a,BT08b}. A recent comprehensive survey can be found in \cite{BOP18} and we also refer the reader to the related work in the context of nonlinear wave equations \cite{BB14,Bringmann18b,Bringmann18,BT14,S11,DLM18,DLM17,LM13,LM16,OP16,Pocovnicu17}. 

Most previous work on probabilistic well-posedness for energy critical nonlinear wave equations has dealt with defocusing nonlinearities, where the \( -u^5 \) in \eqref{nlw} is replaced by \( +u^5 \). For initial data in the energy space, it is well-known that all solutions of the defocusing equation are global and scatter as \( t \rightarrow \pm \). A natural question is whether global well-posedness and scattering are stable under random perturbations of the initial data. More precisely, we assume that the random initial data is of the form
\begin{equation}
 (u_1 + f_1^\omega, u_2 + f_2^\omega),
\end{equation}
where \( (u_1,u_2) \in \dot{H}^1(\R^3) \times L^2(\R^3) \), \( (f_1,f_2) \in H^s(\R^3) \times H^{s-1}(\R^3) \) with \( 0\leq s < 1 \), and \( (f_1^\omega,f_2^\omega) \) is a randomized version of \( (f_1,f_2) \). For an exact definition of the randomization, we refer to Definition \ref{intro:definition_randomization} and Remark \ref{intro:remark_randomization} below. In \cite{Pocovnicu17}, Pocovnicu proved almost sure global well-posedness for the defocusing energy critical nonlinear wave equation in spatial dimensions \( d=4,5 \) for \( s >0 \). In particular, this result lies almost a full derivative below the deterministic threshold. In \cite{OP16}, Oh and Pocovnicu proved the same result in three spatial dimensions under the stronger condition \( s > 1/2 \). While both results yield global solutions, they do not provide much information on the asymptotic behavior as \( t \rightarrow \pm \). The stability of the scattering mechanism under random perturbations was first proved by Dodson, Lührmann, and Mendelson \cite{DLM18,DLM17}. Their result applies in four spatial dimensions and requires the spherical symmetry  condition \( (f_1,f_2) \in H^s_{\rad}(\R^4) \times H^{s-1}_{\rad}(\R^4)\), where \( s > 0 \). This result was extended to the three dimensional case by the author \cite{Bringmann18}, but it requires the stronger condition \( s > 11/12 \). Without the spherical symmetry assumption, almost sure scattering was subsequently proved by the author (with \( d=4 \) and \( s > 11/12 \)) in \cite{Bringmann18c}. Similar results were also obtained for the defocusing energy critical nonlinear Schrödinger equations in \cite{DLM18,KMV19,OOP17}.

Unfortunately, the focusing nonlinear wave equation \eqref{nlw} with random initial data is less understood. One natural question is to consider the stability of the special families of solutions, such as solitons or the ODE-blowup, under random perturbations of the initial data. In \cite{KM19}, Kenig and Mendelson answered this question for random and radial perturbations of the  soliton \( W \). They treat random perturbations in a weighted Sobolev space with regularity \( s > 5/6 \). Due to the unstable direction, however, the random perturbation also has to be projected onto a Lipschitz manifold of codimension one. Inspired by Kenig and Mendelson's result for solitons, the main result of this paper proves the stability of the ODE-blowup under random and radial perturbations. \\

Before we state the main theorem, we define the random initial data \( (f_1^\omega,f_2^\omega) \). 
\begin{definition}[Radial randomization \cite{Bringmann18}]\label{intro:definition_randomization}
Let \( s \in \R \), let \( f \in H^s_{\rad}(\R^3)\), and let \( (X_n)_{n=0}^\infty \)  be a sequence of independent, standard real-valued Gaussians. We define the radial randomization \( f^\omega \) by 
\begin{equation}
\widehat{f^\omega}(\xi) := \sum_{n=0}^\infty X_n(\omega) 1_{[n,n+1)}(\|\xi\|_2) \widehat{f}(\xi). 
\end{equation}
\end{definition}

\begin{remark}\label{intro:remark_randomization}
The radial randomization is based on a decomposition of frequency space into annuli of width one. It first appeared in \cite{Bringmann18} and a similar randomization (using the distorted Fourier transform) was used by Kenig and Mendelson in \cite{KM19}. It is inspired by earlier the Wiener randomization \cite{BOP14,LM13}, which is based on a decomposition of frequency space into unit-scale cubes. We also refer the interested reader to the physical randomization in \cite{Murphy19}, the microlocal randomization in \cite{Bringmann18c}, and a randomization based on good frames in \cite{BK19}. \\
Instead of Gaussian random variables,  it suffices to assume that the sequence \( (X_k)_{k=0}^\infty \) is independent and uniformly sub-gaussian (see Definition \ref{prelim:def_subgaussian}). 
\end{remark}

We now consider the random data Cauchy problem
\begin{equation}\label{rnlw}
\begin{cases} 
- \partial_t^2 u(t,x) + \Delta u(t,x)= - u(t,x)^5 \qquad  (t,x) \in \R \times \R^3, \\
u(0)=u^{(1)}(0)+ f_{1}^\omega, ~ \partial_t u (0)= \partial_t u^{(1)}(0)+ f_{2}^\omega.
\end{cases}
\end{equation}
Here, \( u^{(1)} \) is the ODE-blowup \eqref{intro:ode_blowup} with \( T=1 \) and \( (f_1,f_2)\in H^s_{\rad}(\R^3) \times H^{s-1}_{\rad}(\R^3) \). As in previous work on the stability of the ODE-blowup \cite{Donninger17,DS12}, our main theorem is stated in terms of the evolution inside a light cone. To this end, we define
\begin{equation}
\CT := \{ (t,x) \in [0,T] \times \R^3 \colon \| x\|_2 \leq T-t \}. 
\end{equation}
\newpage
\begin{theorem}\label{main_theorem}
Let \( s > 7/10 \), let \( (f_1,f_2) \in H^s_{\rad}(\R^3) \times H^{s-1}_{\rad}(\R^3) \), let \( 0 < \delta \leq \delta_0 \) be sufficiently small, and let \( c>0 \) be a small absolute constant. With probability greater than or equal to 
\begin{equation}\label{probability}
1- c^{-1} \exp\Big( - c \delta^2 \| (f_1,f_2) \|_{H^s\times H^{s-1}}^{-2} \Big), 
\end{equation}
there exists a (random) blowup time \( T \in [1-\delta,1+\delta] \) and a solution \( u \colon \CT \rightarrow \R \) of \eqref{rnlw} satisfying
\begin{equation}\label{intro:eq_bound}
\| (T-t)^{-\frac{3}{4}} (u - u^{(T)}) \|_{L_t^2 L_x^4(\CT)} + \| u - u^{(T)} \|_{L_t^5 L_x^{10}(\CT)} \lesssim \delta .
\end{equation}
\end{theorem}
\begin{remark}\label{remark:different_norms_theorem}
Using the explicit expression \eqref{intro:ode_blowup}, we see that 
\begin{equation*}
  \| (T-t)^{-\frac{3}{4}} u^{(T)} \|_{L_t^2 L_x^4(\CT)}  = \| u^{(T)} \|_{L_t^5 L_x^{10}(\CT)} = \infty. 
\end{equation*}
Thus, the estimate \eqref{intro:eq_bound} shows that \( u^{(T)} \) and \( u\) agree at the top order. In the deterministic setting, Donninger \cite{Donninger17} also controlled the difference \( u-u^{(T)} \) in \( L_t^2 L_x^\infty \), which is stronger than the weighted \( L_t^2 L_x^4\)-bound. Unfortunately, this bound is not available in our setting, see Remark \ref{prob:rem_probabilistic_gain}. \\
We emphasize that the lower bound on the probability \eqref{probability} is close to one as long as the data is much smaller than \( \delta \). 
\end{remark}

\begin{remark}
In \cite{OOP17}, Oh, Okamoto, and Pocovnicu consider the energy critical nonlinear Schrödinger equation without gauge invariance on \( \R^d \) with \( d=5,6 \), that is,
\begin{equation}\label{non_gauge_nls}
\begin{cases} 
i \partial_t u(t,x) +\Delta u(t,x)= \lambda \, |u(t,x)|^{\frac{d+2}{d-2}} \qquad (t,x) \in \R \times \R^d, \\
u(0,x)= u_0(x), 
\end{cases}
\end{equation}
where \( \lambda \in \C \backslash \{ 0 \} \). In an earlier deterministic work by Ikeda and Inui \cite{II15}, the test function method was used to show that regular initial data satisfying a sign condition and lower bounds (see \cite[(1.5)]{II15}) leads to finite-time blowup of \eqref{non_gauge_nls}. Similar as in Theorem \ref{main_theorem}, \cite{OOP17} shows that a random perturbation of the initial data from \cite{II15} still leads to finite-time blowup. The construction of the blowup, however, is different from the deterministic methods in \cite{Donninger17,DS12} and arguments in this paper. In particular, while \cite{II15,OOP17} prove the existence of finite-time blowup, Theorem \ref{main_theorem} also characterizes the blowup profile. 
\end{remark}

Before the end of this introduction, we present an overview of the rest of the paper. In Section \ref{section:prelim}, we recall a few basic facts from harmonic analysis and probability theory. In Section \ref{section:reformulation}, we perform several changes of variables. We first utilize the Bourgain-Da Prato-Debussche trick \cite{Bourgain96,DD02}, which converts the random and rough initial data in \eqref{rnlw} into a random and rough forcing term. We then switch from Cartesian into similarity coordinates and perturb around the ODE-blowup. Due to the time-translation invariance, the resulting one-parameter semigroup has one unstable mode. As in \cite{Donninger17,DS12}, we therefore first work with a modified Duhamel integral equation and then determine the blowup time through a soft argument. In Section \ref{section:strichartz}, we prove probabilistic Strichartz estimates for the free wave equation in similarity coordinates. We rely on bounds for annular Fourier multipliers on spaces of radial functions and probabilistic Strichartz estimates in Cartesian coordinates from the author's earlier work \cite{Bringmann18}. In \mbox{Section \ref{section:nonlinear}}, we use the probabilistic Strichartz estimates to solve the nonlinear problem. 

An earlier version of this paper \cite{Bringmann20} contained a much longer proof of Theorem \ref{main_theorem}. We postpone a more detailed comparison until Section \ref{section:comparison}, but already mention that the main difference lies in the order of using the Bourgain-Da Prato-Debussche trick and perturbing around the ODE-blowup. \\

\textbf{Acknowledgements:} I would like to thank my advisor Terence Tao for his guidance and support. I would also like to thank Benjamin Harrop-Griffiths, Joachim Krieger, Redmond McNamara, Dana Mendelson, and Tadahiro Oh for helpful discussions. 

\section{Notation and Preliminaries}\label{section:prelim}

If \( A,B \geq 0 \), we write \( A \lesssim B \) if there exists an absolute constant \( C >0 \) such that \( A \leq CB \). We also write \( A \sim B \) if \( A \lesssim B \) and \( B \lesssim A \). For any \( d \geq 1 \) and \( x \in \R^d \), we set \( \langle x \rangle := (1+\| x\|_2^2 )^{\frac{1}{2}}\). 

\subsection{Fourier analysis}
For any dimension \( d \geq 1 \) and any Schwartz function \( f \in S(\R^d) \), we define its Fourier transform \( \widehat{f} \colon \R^d \rightarrow \C \) by
\begin{equation}
\mathcal{F}(f)(\xi)=\widehat{f}(\xi):= \frac{1}{(2\pi)^{\frac{d}{2}}} \int_{\R^d} e^{-i\xi x} f(x) \dx. 
\end{equation}
The Fourier inversion formula then implies that 
\begin{equation}
f(x) = \frac{1}{(2\pi)^{\frac{d}{2}}} \int_{\R^d} e^{i\xi x} \widehat{f}(\xi) \dxi. 
\end{equation}
If \( f \)  is spherically symmetric, it follows from the relation between the Fourier and Hankel transforms that 
\begin{equation}
\nu^{\frac{d-2}{2}} \widehat{f}(\nu) := \int_0^\infty r^{\frac{d}{2}} J_{\frac{d-2}{2}}(\nu r) f(r) \dr  \qquad \text{and} \qquad r^{\frac{d-2}{2}} f(r) =\int_0^\infty \nu^{\frac{d}{2}} J_{\frac{d-2}{2}}(\nu r) \widehat{f}(\nu) \dnu. 
\end{equation}
Here, \( r:= \| x\|_2 \), \( \nu:= \| \xi\|_2 \), and \( J_v(\cdot) \) denotes the Bessel function of the first kind. As a special case, we obtain for all radial functions \(f \in S(\R^3) \) in three dimensions that
\begin{equation}\label{pre:eq_bessel_transform}
r f(r) = \sqrt{\frac{2}{\pi}} \int_0^\infty \sin(r\nu) \widehat{f}(\nu) \nu \,\mathrm{d}\nu \qquad \text{and} \qquad \nu \widehat{f}(\nu) = \sqrt{\frac{2}{\pi}}\int_0^\infty \sin(r\nu) f(r) r \, \mathrm{d}r. 
\end{equation}
 
Using the Fourier transform, we define for any \( s \in \R \) the fractional derivative operator \( \langle \nabla \rangle^s \) by 
\begin{equation}
\mathcal{F}(\langle \nabla \rangle^s f)(\xi) := \langle \xi \rangle^s \mathcal{F}(f)(\xi). 
\end{equation}
The fractional Sobolev spaces \( H^s(\R^d) \) are defined by completion of Schwartz space with respect to the norm
\begin{equation*}
\| f \|_{H^s(\R^d)} := \| \langle \nabla \rangle^s f \|_{L^2(\R^d)}. 
\end{equation*}
To simplify our notation, we further set \( \mathcal{H}^s(\R^d):= H^s(\R^d) \times H^{s-1}(\R^d) \). Finally, we define the Littlewood-Paley projections \( \{ Q_N \}_{N\in 2^{\mathbb{N}_0}} \)  as follows: We let \( \chi\colon \R^d \rightarrow [0,1] \) be a smooth cut-off function 
which equals one on \( \| x\|_2 \leq 1/2 \) and zero on \( \| x\|_2 \geq 1 \). We then define
\begin{equation}
\chi_1(\xi):= \chi(\xi) \quad \text{and} \quad \chi_N(\xi):= \chi\Big( \frac{\xi}{N}\Big) - \chi\Big( \frac{2\xi}{N}\Big),~\text{where}~ N\geq 2. 
\end{equation}
For any \( N \geq 1 \), we then define the Littlewood-Paley projection \( Q_N \) by 
\begin{equation}
\mathcal{F}(Q_N f)(\xi):= \chi_N(\xi) \mathcal{F}(f)(\xi). 
\end{equation}
We choose the letter \( Q \) for the Littlewood-Paley projections, instead of the more conventional choices \( P \) or \( S \), since \( P \) already denotes the projection on the unstable mode and \( S \) denotes the one-parameter semigroup. 
\subsection{Probability theory}
We recall the basic properties of sub-gaussian random variables. The organization follows a similar subsection in \cite{Bringmann18c} and we refer the reader to \cite{Vershynin} for a more detailed introduction. 

\begin{definition}[Sub-gaussian random variable]\label{prelim:def_subgaussian}
Let \( (\Omega,\mathcal{F},\mathbb{P}) \) be a probability space and let \( X \colon (\Omega,\mathcal{F}) \rightarrow \R \) be a random variable. We define 
\begin{equation}
\| X\|_{\psi_2} := \sup_{r\geq 1} \frac{\big( \E |X|^r \big)^{\frac{1}{r}}}{\sqrt{r}}.
\end{equation}
We call \( X \) sub-gaussian if \( \| X \|_{\psi_2} < \infty \). We call a family of random variables \( \{ X_j \}_{j\in J} \) uniformly sub-gaussian if \( \sup_{j\in J} \| X_j \|_{\psi_2} < \infty \). 
\end{definition}
The connection with the Gaussian distribution is most easily seen from the following lemma.

\begin{lemma}[{Tail estimate, \cite[Proposition 2.52]{Vershynin}}]\label{prelim:lemma_tail}
Let \(X \) be a sub-gaussian random variable. Then, we have for all \( \lambda \geq 0 \) that
\begin{equation}
P( |X|\geq \lambda) \leq 2 \exp\Big( - c \frac{\lambda^2}{\| X\|_{\psi_2}^2}\Big). 
\end{equation}
\end{lemma}

We also recall Khintchine's inequality, which is a concentration-inequality for sums of independent uniformly sub-gaussian random variables.

\begin{lemma}[{Khintchine's inequality, cf. \cite[Proposition 2.52]{VershyninRMT}}]
Let \( (X_j)_{j=1,\hdots,J} \) be a finite sequence of independent sub-gaussian random variables with zero mean and let \( (a_j)_{j=1,\hdots,J} \) be a finite sequence of real or complex numbers. Then, we have for all \( r \geq 1 \) that 
\begin{equation}
\Big( \E |\sum_{j=1}^J a_j X_j |^r \Big)^{\frac{1}{r}} \lesssim \sqrt{r} \big( \max_{1\leq  j \leq J} \| X_j \|_{\psi_2} \big) \Big( \sum_{j=1}^J |a_j|^2 \Big)^{\frac{1}{2}}. 
\end{equation}
In other words, it holds that 
\begin{equation*}
\Big\| \sum_{j=1}^J a_j X_j \Big\|_{\psi_2} \lesssim \big( \max_{1\leq  j \leq J} \| X_j \|_{\psi_2} \big) \| a_j \|_{\ell^2_j}. 
\end{equation*}
\end{lemma}
Whereas Khintchine's inequality controls the sub-gaussian norm of a random series, we also record the following estimate for the maximum of sub-gaussian random variables.

\begin{lemma}[Maximum of sub-gaussian random variables]\label{prelim:lemma_maximum}
Let \( (X_j)_{j=1,\hdots,J} \) be a finite sequence of (not necessarily independent) sub-gaussian random variables. Then, it holds that 
\begin{equation}
\| \max_{1\leq  j \leq J} |X_j| \|_{\psi_2} \lesssim \sqrt{\log(2+J)} \max_{1\leq  j \leq J} \| X_j \|_{\psi_2}. 
\end{equation}
\end{lemma}
\begin{proof}
Let \( r \geq 1 \) be arbitrary and let \( p \geq 1 \) remain to be chosen. From the embedding \( \ell^p \hookrightarrow \ell^\infty \) and Hölder's inequality, we obtain that
\begin{align*}
& \Big(\E \max_{1\leq  j \leq J} |X_j|^r \Big)^{\frac{1}{r}} = \Big( \E \big( \sum_{j=1}^J |X_j|^{pr}\big)^{\frac{1}{p}}\Big)^{\frac{1}{r}} \lesssim \Big( \E \sum_{j=1}^J |X_j|^{pr} \Big)^{\frac{1}{pr}} \leq \Big( \sum_{j=1}^J (\sqrt{pr} \| X_j \|_{\psi_2})^{pr} \Big)^{\frac{1}{pr}} \\
&\lesssim \sqrt{pr} J^{\frac{1}{pr}} \max_{1\leq  j \leq J} \| X_j \|_{\psi_2} \lesssim \sqrt{r} \sqrt{p} J^{\frac{1}{p}} \max_{1\leq j \leq J} \|X_j \|_{\psi_2} . 
\end{align*}
The desired estimate then follows by choosing \( p:= \log(2+J) \). 
\end{proof}

We now record the following large-deviation estimate for the radial randomization \( f^\omega \) in Sobolev spaces. Similar estimates for the Wiener randomization are well-known in the literature on dispersive equations with random initial data. 
\begin{lemma}[The $H^s$-norm of the radial randomization]\label{prelim:lemma_hs_norm}
Let \( s \in \R\), let \( f \in H^s_{\rad}(\R^3) \), and let \( f^\omega \) be as in Definiton \ref{intro:definition_randomization}. Then, it holds for all \( r \geq 1 \) that 
\begin{equation}\label{prelim:eq_hs_norm}
\| f^\omega \|_{L_\omega^r H^s_{\rad}(\Omega \times \R^3)} \lesssim \sqrt{r} \| f \|_{H^s_{\rad}(\R^3)}. 
\end{equation}
Furthermore, let \( s^\prime > s \) and assume that \( f \not \in H^{s^\prime}_{\rad}(\R^3)\). Then, it holds that 
\begin{equation}\label{prelim:eq_hs_no_gain}
\| f^\omega \|_{H^{s^\prime}_{\rad}(\R^3)} = \infty \quad \text{a.s.}
\end{equation}
\end{lemma}
This lemma shows that the radial randomization does not change the regularity of \( f \) on the scale of \( L^2\)-based Sobolev spaces. 

\begin{proof} Since the radial randomization commutes with the Fourier multiplier \( \langle \nabla \rangle^{s} \), we may assume that \( s=0\). Using Minkowski integral and Khintchine's inequality, we obtain that for all \( r\geq 2 \) that
\begin{align*}
&\| f^\omega \|_{L_\omega^r L_x^2(\Omega \times \R^3)} \leq \| f^\omega \|_{L_x^2 L_\omega^r(\R^3 \times \Omega)} \lesssim \sqrt{r} \| 1_{[n,n+1]}(|\nabla|) f \|_{L_x^2 \ell_n^2(\R^3 \times \mathbb{N})} \\
&= \sqrt{r} \| 1_{[n,n+1]}(|\nabla|) f \|_{\ell_n^2L_x^2 (\mathbb{N} \times \R^3)} = \sqrt{r} \| f \|_{L^2_x(\R^3)}. 
\end{align*}
This yields \eqref{prelim:eq_hs_norm}. A simple calculation shows that 
\begin{equation}
\E \exp( - \| f^\omega \|_{H^{s^\prime}_\rad(\R^3)}^2) = 0, 
\end{equation}
which yields \eqref{prelim:eq_hs_no_gain}. 
\end{proof}

\section{Bourgain-Da Prato-Debussche trick, similarity coordinates,  and first-order systems} \label{section:reformulation}

In this section, we perform several standard reformulations of the Cauchy problem \eqref{nlw}. They consist of a combination of  the Bourgain-Da Prato-Debussche trick \cite{Bourgain96,DD02} with the first-order systems from \cite{Donninger17,DS12}.

\subsection{Bourgain-Da Prato-Debussche trick}

We cannot directly use a contraction argument to solve the nonlinear wave equation \eqref{rnlw}, since the random initial data \( (f_1^\omega,f_2^\omega) \) only lives in scaling-supercritical Sobolev spaces. To overcome this difficulty, we first extract the linear evolution of the random initial data, which is known as Bourgain's trick \cite{Bourgain96} in the dispersive PDE literature and the Da Prato-Debussche trick \cite{DD02} in the SPDE literature\footnote{Strictly speaking, Bourgain \cite{Bourgain96} worked with random initial data and Da Prato-Debussche \cite{DD02} worked with a stochastic forcing term. Thus, our setting may be a bit closer to Bourgain's work \cite{Bourgain96}, but we still chose the terminology Bourgain-Da Prato-Debussche trick.}. To this end, we let
\begin{equation}\label{re:eq_fomega}
f^\omega(t,x):= \cos(t|\nabla|) f_1^\omega(x) + \frac{\sin(t|\nabla|)}{|\nabla|} f_2^\omega (x)
\end{equation}
be the solution of the linear wave equation with random initial data. We then decompose \( u= f^\omega+v\) and obtain the nonlinear wave equation 
\begin{equation}\label{re:eq_v}
\begin{cases} 
- \partial_t^2 v + \Delta v = - (v+f^\omega)^5 \qquad  (t,x) \in \R \times \R^3, \\
v(0)=u^{(1)}(0) , \quad \partial_t v (0)= \partial_t u^{(1)}(0),
\end{cases}
\end{equation}
for the nonlinear component \( v \). As can be seen by comparing \eqref{rnlw} and \eqref{re:eq_v}, we have replaced the rough initial data by a rough forcing term. In our setting, this is a favorable trade-off. The rough forcing term can eventually  be controlled through the smoothing effect of the Duhamel integral and probabilistic Strichartz estimates. From Proposition \ref{prob:prop_main} below, we see that 
\begin{equation}
\| (T-t)^{-\frac{3}{4}} f^\omega(t,x) \|_{L_t^2 L_x^4(\CT)} + \| f^\omega(t,x) \|_{L_t^5 L_x^{10}(\CT)} \lesssim \delta
\end{equation}
with high probability and hence our main theorem reduces to an estimate of \( v-u^{(T)} \).

\subsection{Similarity coordinates and first-order systems}\label{section:similarity}

Since the initial data and \( f^\omega \) in \eqref{re:eq_v} are spherically symmetric, we can rewrite the nonlinear wave equation as
\begin{equation}\label{re:eq_v_radial}
- \partial_t^2 v(t,r) + \partial_r^2 v(t,r) + \frac{2}{r} \partial_r v(t,r)= - (v(t,r)+f^\omega(t,r))^5. 
\end{equation}
We now switch from Cartesian to similarity coordinates. For a fixed  \( T \in [1/2,3/2] \) and \( (t,x) \in \CT \), we define the similarity coordinates \( (\tau,y) \in [0,\infty) \times \B^3 \) by
\begin{equation}\label{re:eq_similarity}
\tau := - \log(T-t) + \log(T) \qquad \text{and} \qquad y:= \frac{x}{T-t}
\end{equation}
We also write \( \rho := |y| = r/(T-t) \). In similarity coordinates, we write the solution of \eqref{re:eq_v} as 
\begin{equation}\label{re:eq_psi}
\psi(\tau,\rho;T):= (T-t)^{\frac{1}{2}} v(t,r). 
\end{equation}
To simplify the notation, we will often omit the dependence of \( \psi(\tau,\rho;T) \) on \( T \) and simply write \( \psi(\tau,\rho)= \psi(\tau,\rho;T) \). The ODE-blowup  \( u^{(T)}\colon \CT \rightarrow \R \), which is only a solution of \eqref{re:eq_v_radial} if \( f^\omega\equiv 0 \),  is given in similarity coordinates by 
\begin{equation}
\psi^{(T)}(\tau,\rho;T) = \kappa . 
\end{equation}
As in \cite{Donninger17}, we obtain a first-order system by introducing
\begin{equation}
\begin{aligned}
\psi_1 (\tau,\rho) &:= \psi(\tau,\rho), \\
\psi_2 (\tau,\rho) &:= \big(\partial_\tau + \rho \partial_\rho + \frac{1}{2}\big) \psi(\tau,\rho). 
\end{aligned} 
\end{equation}
The purpose of the lower-order term \( 1/2 \cdot \psi(\tau,\rho) \) in the definition of \( \psi_2(\tau,\rho) \) is to cancel the effect of the factor \( (T-t)^{\frac{1}{2}} \) in \eqref{re:eq_psi} on the initial data. In the unknowns \( (\psi_1,\psi_2) \), the nonlinear wave equation \eqref{nlw} takes the form 
\begin{equation}
\begin{cases}
\partial_\tau \psi_1 &= - \rho \partial_\rho \psi_1 - \frac{1}{2} \psi_1 + \psi_2 \\
\partial_\tau \psi_2 &= \partial_\rho^2 \psi_1 + \frac{2}{\rho} \partial_\rho \psi_1 - \rho \partial_\rho \psi_2 - \frac{3}{2} \psi_2 + (\psi_1+f^{T,\omega})^5 \\
\psi_1(0,\rho)  &= T^{\frac{1}{2}} u^{(1)}(0,T\rho), \quad \psi_2(0,\rho)= T^{\frac{3}{2}} \partial_t u^{(1)}(0,T\rho). 
\end{cases}
\end{equation}
With a slight abuse of notation, we wrote 
\begin{equation}\label{re:eq_f_similarity}
f^{T,\omega}(\tau,\rho):= (T-t)^{\frac{1}{2}} f^\omega(t,x). 
\end{equation}
The superscript \( T \) indicates both the change into similarity coordinates and the blowup time in \eqref{re:eq_similarity}. From this derivation, it is clear that the solution \( u \) of \eqref{rnlw} is given in similarity coordinates by \( f^{T,\omega}+\psi_1 \). 
We write \( \psib = (\psi_1,\psi_2) \) for the vector containing both components. From the definition of \((\psi_1,\psi_2) \), it follows that the ODE-blowup \( u^{(T)} \) corresponds to \( \psib^{(T)}=(\kappa,\kappa/2) \). 
Finally, we decompose the evolution into the ODE-blowup and a lower-order term. To this end, we set 
\begin{equation}\label{re:eq_phi}
\phib(\tau,\rho;T) := \psib(\tau,\rho;T) -  \psib^{(T)}(\tau,\rho;T) = ( \psi_1(\tau,\rho;T) - \kappa , \psi_2(\tau,\rho;T) - \kappa/2 ) . 
\end{equation}
We denote the individual components of \( \phib \) by \( (\phi_1,\phi_2) \). Then, the nonlinear wave equation \eqref{nlw} in terms of \( (\phi_1,\phi_2) \) is given by 
\begin{equation}\label{re:eq_phi_nlw}
\begin{cases}
\partial_\tau \phi_1 &= - \rho \partial_\rho \phi_1 - \frac{1}{2} \phi_1 + \phi_2 \\
\partial_\tau \phi_2 &= \partial_\rho^2 \phi_1 + \frac{2}{\rho} \partial_\rho \phi_1 - \rho \partial_\rho \phi_2 - \frac{3}{2} \phi_2 + (\kappa+f^{T,\omega}+\phi_1)^5 - \kappa^5, \\
\phi_1(0,\rho)  &= T^{\frac{1}{2}} u^{(1)}(0,T\rho)- \kappa, \quad \phi_2(0,\rho)= T^{\frac{3}{2}} \partial_t u^{(1)}(0,T\rho) - \frac{\kappa}{2}. 
\end{cases}
\end{equation}
We expand the quintic power and decompose
\begin{equation*}
(\kappa+ f^{T,\omega}+\phi_1)^5 -\kappa^5 = 5\kappa^4 \phi_1 + N(\phi_1,f^{T,\omega}), 
\end{equation*}
where 
\begin{equation}\label{re:eq_N}
\begin{aligned}
N(\phi_1,f^{T,\omega})&=  5 [(\kappa+f^{T,\omega})^4-\kappa^4] \,  \phi_1 + 10 (\kappa+f^{T,\omega})^3 \phi_1^2 + 10 (\kappa+f^{T,\omega})^2 \phi_1^3  \\
&+ 5  (\kappa+f^{T,\omega}) \phi_1^4 + \phi^5_1 + (\kappa+f^{T,\omega})^5 -\kappa^5. 
\end{aligned}
\end{equation}
In particular, elementary estimates lead to 
\begin{equation}\label{re:eq_estimate_N}
|N(\phi_1,f^{T,\omega})| \lesssim |f^{T,\omega}|+|f^{T,\omega}|^5 + |\phi_1|^2 +|\phi_1|^5. 
\end{equation}
We point out that \eqref{re:eq_N} contains terms which are linear in \( \phi_1 \), such as \( 20 \kappa^3 f^{T,\omega} \phi_1 \). Due to the decay of \( f^{T,\omega} \) in similarity coordinates from Proposition \ref{prob:prop_main}, they can still be treated perturbatively. In contrast, the linear term \( 5 \kappa^4 \phi_1 \), which is a result of the interaction between \( \phi_1 \) and the ODE-blowup, has to be included in the linear evolution of \( \phi_1 \). We further emphasize that \( N(\phi_1,f^{T,\omega}) \) contains the potentially dangerous term  \( 5\kappa^4 f^{T,\omega}\), but postpone a more detailed discussion of this until Section \ref{section:comparison}. Using this notation and the value of \( \kappa \), we can rewrite \eqref{re:eq_phi_nlw} as 
\begin{equation}\label{re:eq_phi_N_nlw}
\begin{cases}
\partial_\tau \phi_1 &= - \rho \partial_\rho \phi_1 - \frac{1}{2} \phi_1 + \phi_2 \\
\partial_\tau \phi_2 &= \partial_\rho^2 \phi_1 + \frac{2}{\rho} \partial_\rho \phi_1 - \rho \partial_\rho \phi_2 - \frac{3}{2} \phi_2 + \frac{15}{4} \phi_1 + N(\phi_1,f^{T,\omega}), \\
\phi_1(0,\rho)  &= T^{\frac{1}{2}} \partial_t u^{(1)}(0,T\rho)- \kappa, \quad \phi_2(0,\rho)= T^{\frac{3}{2}} \partial_t u^{(1)}(0,T\rho) - \frac{\kappa}{2}. 
\end{cases}
\end{equation}

\subsection{The linearized problem}\label{section:linearized}

Since the contribution of \( N(\phi_1,f^{T,\omega}) \) will be treated perturbatively, we are also interested in the linear evolution without this term. We consider 
\begin{equation}\label{re:eq_phi_linear}
\begin{cases}
\partial_\tau \phi_1 &= - \rho \partial_\rho \phi_1 - \frac{1}{2} \phi_1 + \phi_2 \\
\partial_\tau \phi_2 &= \partial_\rho^2 \phi_1 + \frac{2}{\rho} \partial_\rho \phi_1 - \rho \partial_\rho \phi_2 - \frac{3}{2} \phi_2 + \frac{15}{4} \phi_1\\
\phib(0)&= \phib_0,
\end{cases}
\end{equation}
which corresponds to a linear wave equation with a potential. 
We now recall some notation and basic properties regarding \eqref{re:eq_phi_linear} from \cite{Donninger17,DS12}. We define the differential operator
\begin{equation}
\Lzt \phib(\rho) := \begin{pmatrix} - \rho \partial_\rho \phi_1 - \frac{1}{2} \phi_1 + \phi_2 \\ \partial_\rho^2 \phi_1 + \frac{2}{\rho} \partial_\rho \phi_1 - \rho \partial_\rho \phi_2 - \frac{3}{2} \phi_2 \end{pmatrix}
\end{equation}
with domain 
\begin{equation}
\D (\Lzt) :=C^2([0,1])\times C^1([0,1]).
\end{equation}
We note that despite the singularity, \( \rho^{-1} \partial_\rho \phi_1 \) is still contained in \( L^2(\B^3) \) for all \( \phib \in \D (\Lzt) \). Since \( \Lzt \) does not contain the potential term \( 15/4 \cdot \phi_1 \), it corresponds to the free wave equation in similarity coordinates. For any \( \phib \in C^1([0,1])\times C^0([0,1]) \), we define 
\begin{equation}\label{re:eq_hone_norm}
\| \phib \|_{\Hone}^2 := \int_0^1 ((\rho \phi_1)^\prime)^2 + (\rho \phi_2)^2 \drho.
\end{equation}
The Hilbert space \( \Hone \) is then defined through completion. A simple calculation shows that 
\begin{equation}\label{re:eq_hone}
\| \phib \|_{\Hone}^2 = \int_0^1 ((\phi_1^\prime)^2+\phi_2^2) \rho^2 \drho + \phi_1^2(0),
\end{equation}
and hence \( \Hone \simeq H^1_\rad(\B^3) \times L^2_\rad(\B^3) \). From \cite[Proposition 2.1]{Donninger17}, it follows that \( \Lzt \) has a closed extension \( \Lz \), which generates a strongly-continuous and uniformly bounded one-parameter semi-group \( \{ S_0(\tau)\colon \tau \geq  0 \}\). We define the compact linear operator \( \Lp \colon \Hone \rightarrow \Hone \) by 
\begin{equation}
\Lp \phib(\rho) := \begin{pmatrix} 0 \\ \frac{15}{4} \phi_1(\rho) \end{pmatrix}.
\end{equation}
Finally, we set \( L:= \Lz + \Lp \), which is (an extension of) the formal differential operator in \eqref{re:eq_phi_linear}. Using the bounded perturbation theorem, it follows that \( L \) generates a strongly continuous one-parameter semigroup \( S(\tau) \). We also recall the following properties of \( L \) and the associated one-parameter semigroup \( S(\tau) \). 

\begin{lemma}[{\cite[Proposition 2.3 and Lemma 5.7]{Donninger17}}]\label{re:lem_semigroup}
We have that \( \sigma(L) \subseteq \{ z \in \mathbb{C} \colon \operatorname{Re} z \leq 0 \} \medcup \{ 1 \} \) and \( 1 \in \sigma_p(L) \). The geometric eigenspace of the eigenvalue \( 1 \) is one-dimensional and 
spanned by 
\begin{equation}
\g (\rho) := \begin{pmatrix} 2 \\ 3 \end{pmatrix}. 
\end{equation}
Furthermore, there exists a bounded projection \( P \colon \Hone \rightarrow \langle g \rangle \) such that \( [P,S(\tau)] \)=0 for all \( \tau\geq 0 \). As a consequence, we have that \( S(\tau) P = e^\tau P \) for all \( \tau \geq 0 \). Finally, we have that 
\begin{equation}
\| S(\tau) (I-P) \|_{\Hone \rightarrow \Hone} \lesssim 1 
\end{equation}
uniformly for all \( \tau \geq 0\). 
\end{lemma}

The unstable mode \( \g \) corresponds to the time-translation invariance of \eqref{nlw}. As a result, it does not correspond to a physical instability but is instead an artifact of working with a fixed \( T>0 \) in \eqref{re:eq_similarity}. We also recall the following deterministic Strichartz estimates from \cite{Donninger17}. 

\begin{proposition}[{\cite[Theorem 4.1]{Donninger17}}]\label{re:prop_deterministic_strichartz}
Let  \( 2\leq q,p \leq \infty \) satisfy the scaling condition \( \frac{1}{q}+\frac{3}{p}= \frac{1}{2} \). Then, we have the bound
\begin{equation}
\| (S(\tau) (I-P) \fb_0)_1\|_{L_\tau^q L_y^p([0,\infty)\times \B^3))} \lesssim \| (I-P) \fb_0 \|_{\Hone} 
\end{equation}
for all \( \fb_0 \in \Hone \). In addition, we have 
\begin{equation}
\Big \| \int_0^\tau \big(S(\tau-\sigma) (I-P) \hb (\sigma) \big)_1 \dsigma  \Big\|_{L_\tau^q L_y^p( [0,\infty) \times \B^3)} \lesssim \| (I-P) \hb \|_{L_\tau^1 \mathcal{H}_\rad^1([0,\infty)\times \B^3)} 
\end{equation}
for all \( \hb \in C_\tau \mathcal{H}_\rad^1([0,\infty)\times \B^3) \medcap L_\tau^1 \mathcal{H}_\rad^1([0,\infty)\times \B^3) \). 
\end{proposition}

\subsection{The Duhamel integral}\label{section:duhamel}
By using the one-parameter semi-group \( S(\tau) \), we can rewrite \eqref{re:eq_phi_N_nlw} in integral form as 
\begin{equation}\label{re:eq_duhamel}
\phib(\tau) = S(\tau) \phib_0^T + \int_0^\tau S(\tau-\sigma) \mathcal{N}( \phi_1, f^{T,\omega}) (\sigma) \dsigma,
\end{equation}
where
\begin{equation}
\phib_0^T  = \begin{pmatrix} T^{\frac{1}{2}} u^{(1)}(0,T\cdot) - \kappa \\ T^{\frac{3}{2}} \partial_t u^{(1)}(0,T\cdot) - \frac{\kappa}{2} \end{pmatrix} \qquad \text{and} \qquad  \mathcal{N}( \phi_1, f^{T,\omega})  = \begin{pmatrix} 0 \\ N(\phi_1,f^{T,\omega}) \end{pmatrix}.
\end{equation}
Due to the unstable mode \( \g \), however, we cannot use a contraction argument to solve \eqref{re:eq_duhamel}. Recall that the reason for this instability is that we have not determined the blowup time \( T \) yet. To circumvent this prolem, the following two-step procedure was used in \cite{Donninger17,DS12}: 
\begin{itemize}[leftmargin=0.8cm]
\item[(i)] Solve a modified version of \eqref{re:eq_duhamel} in which the unstable mode has been removed.
\item[(ii)] Choose the blowup time \( T \) so that the modified and original Duhamel integrals coincide.
\end{itemize}
In the remainder of this subsection, we only formulate the modified version of \eqref{re:eq_duhamel}. The contraction argument and the choice of the blowup time \( T \) are postponed until Section \ref{section:nonlinear}. \\
We first split the right-hand side of \eqref{re:eq_duhamel} by using \( I = (I-P)+ P \). We further decompose the unstable component by writing 
\begin{align*}
&P \Big[ S(\tau) \phib_0^T + \int_0^\tau S(\tau-\sigma) \nc \dsigma \Big] \\
&= e^{\tau} P \phib_0^T + \int_0^\tau e^{\tau-\sigma} P \nc \dsigma \\
&= e^{\tau} P \Big[ \phib_0^T + \int_0^\infty e^{-\sigma} \nc \dsigma \Big] - \int_\tau^\infty e^{\tau-\sigma} P \nc \dsigma. 
\end{align*}
Using this decomposition, the two-step procedure takes the following form: 
\begin{itemize}[leftmargin=0.8cm]
\item[(i)] For each time \( T \in [\frac{1}{2},\frac{3}{2}]\), solve the integral equation 
\begin{equation*}
\phib(\tau) = S(\tau) (1-P) \phib_0^T + \int_0^\tau S(\tau-\sigma) (1-P) \nc \dsigma - \int_\tau^\infty e^{\tau-\sigma} P \nc \dsigma. 
\end{equation*}
\item[(ii)] Choose the blowup time \( T \in [\frac{1}{2},\frac{3}{2}] \) such that 
\begin{equation*}
P \Big[ \phib_0^T + \int_0^\infty e^{-\sigma} \nc \dsigma \Big]  = 0. 
\end{equation*}
\end{itemize}
\subsection{Order of the different changes of variables} \label{section:comparison}
Before we continue our discussion of the ODE-blowup, it is instructive to briefly discuss the stability (along a Lipschitz manifold of codimension one) of solitons as in \cite{Beceanu14,KM19,KNS15}. In \cite{KM19}, Kenig and Mendelson study the random data Cauchy problem
\begin{equation}\label{re:eq_soliton_cauchy}
\begin{cases} 
- \partial_t^2 u(t,x) + \Delta u(t,x)= - u(t,x)^5 \qquad  (t,x) \in \R \times \R^3, \\
u(0)=u_1^\omega, ~ \partial_t u (0)=  u_2^\omega,
\end{cases}
\end{equation}
where \( (u_1^\omega,u_2^\omega) \) is a random perturbation (along a Lipschitz manifold of codimension one) of the ground state \( (W,0) \), cf. \cite[Theorem 1.5 and Definition 4.1]{KM19}. After setting \( W_a(x):= a^{\frac{1}{4}} W(a^{\frac{1}{2}}x) \), the authors first linearize the evolution around a modulated soliton by writing 
\begin{equation}
u(t,x):= W_{a(t)}(x) + v(t,x). 
\end{equation}
This leads to a new evolution equation for the nonlinear component \( v \), which is given by
\begin{equation}\label{re:eq_soliton_cauchy_2}
\begin{cases} 
- \partial_t^2 v + ( \Delta + 5 W^4) v = \partial_t^2 W_{a(t)} + 5 (W^4-W_{a(t)}^4) v + N(v,W_{a(t)}) \qquad  (t,x) \in \R \times \R^3, \\
v(0)=u_1^\omega-W, ~ \partial_t u (0)=  u_2^\omega-\dot{a}(0) \partial_a W_a . 
\end{cases}
\end{equation}
Then, Kenig and Mendelson apply the Bourgain-Da Prato-Debussche trick to \eqref{re:eq_soliton_cauchy_2}, which requires probabilistic Strichartz estimates for the linear wave equation with potential 
\begin{equation*}
-\partial_t^2 F + ( \Delta + 5 W^4) F=0. 
\end{equation*}
Due to the potential \( 5W^4\), the probabilistic Strichartz estimate require delicate kernel estimates and form one of the main contributions of \cite{KM19}. \\

In an earlier version of this paper \cite{Bringmann20}, we used a similar approach to prove the stability of the ODE-blowup under random pertubations. We first switched into similarity coordinates, then linearized around the ODE-blowup, and finally used the Bourgain-Da Prato-Debussche trick. As in the soliton-setting, our earlier approach required probabilistic Strichartz estimate for the one-parameter semigroup \( S(\tau)\). They were obtained using delicate oscillatory integral estimates, which formed the main technical contribution of \cite{Bringmann20}. In the current version of this paper, however, we perform the changes of variables in a different order. We first use the Bourgain-Da Prato-Debussche trick, then switch into similarity coordinates, and finally perturb around the ODE-blowup. As a result, this only requires probabilistic Strichartz estimates for the propagator of the free wave equation instead of \( S(\tau) \). In particular, while we still require the (deterministic) Strichartz estimates for \( S(\tau)\) from \cite{Donninger17} to close the nonlinear argument, they can now be used as a blackbox. In the earlier version of the argument, we had to revisit parts of their proof. 

In the current approach, however, one picks up an additional term, which was not present in the earlier version. It corresponds to the linear term \( 5 \kappa^4 f^{T,\omega} \) in \eqref{re:eq_N}, and we therefore require that with high probability 
\begin{equation}\label{re:eq_compare}
\| f^{T,\omega} \|_{L_\tau^1 L_y^2([0,\infty)\times \B^3)} = \| (T-t)^{-2} f^\omega \|_{L_t^1 L_x^2(\CT)} \lesssim \delta. 
\end{equation}
In many problems involving dispersive partial differential equations, such as the stability of solitons, the analogue of \eqref{re:eq_compare} fails. This is often a consequence of the spatial and time-translation invariance of the equation. 
In our setting, however, the spatial and time-translation invariance are (partially) broken due to the compactness of the cone \( \CT \). In fact, a simple application of Sobolev embedding already yields for all \( T\in [\frac{1}{2},\frac{3}{2}] \) and \( s > 3/2 \) that 
\begin{equation*}
\| (T-t)^{-2} e^{\pm it |\nabla|} f_1 \|_{L_t^1 L_x^2(\CT)} \lesssim \| (T-t)^{-2} \|_{L_t^1 L_x^2(\CT)} \| e^{\pm it|\nabla|} f_1 \|_{L_t^\infty L_x^\infty(\R\times \R^3)} \lesssim  \| f_1 \|_{H_x^{s}(\R^3)}. 
\end{equation*}
A suitable probabilistic refinement (Proposition \ref{prob:prop_main}) then leads to \eqref{re:eq_compare}. \\

\section{Probabilistic Strichartz estimates in similarity coordinates}\label{section:strichartz}

In this section, we prove probabilistic Strichartz estimates for the linear wave equation in similarity coordinates.

\begin{proposition}\label{prob:prop_main}
Let \( \fb_0 =(f_1,f_2)  \) be spherically symmetric, let \( f^{T,\omega} \) be as in \eqref{re:eq_f_similarity}, \( 1 \leq r < \infty\), and let \( 1 \leq q,p \leq \infty \). We have the probabilistic Strichartz estimate
\begin{equation}\label{prob:eq_main}
\big\| \sup_{T\in [\frac{1}{2},\frac{3}{2}]} \| f^{T,\omega} \|_{L_\tau^q L_y^p([0,\infty)\times \B^3)} \big \|_{L_\omega^r(\Omega)} \lesssim_{q,p,s} \sqrt{r} \| \fb_0 \|_{\Hs}. 
\end{equation}
under either of the following two conditions: 
\begin{flalign}
\frac{1}{q}+\frac{3}{p} = \frac{1}{2} \qquad &\text{and} \qquad s > 1 - \frac{1}{q} - \frac{3}{p} + \frac{1}{\min(p,q)} & \label{prob:eq_condition_1}\\
\hspace{5ex} \text{or} \hspace{20ex} p \leq 6 \qquad &\text{and} \qquad s > \frac{2}{3}. \label{prob:eq_condition_2}&
\end{flalign}
Furthermore, if \( \fb_0 \in \mathcal{H}^{\frac{7}{10}+}_{\rad}(\R^3) \), then the map
\begin{equation}\label{prob:eq_continuity}
\Big[\frac{1}{2},\frac{3}{2} \Big]\rightarrow (L_\tau^1 L_y^2\medcap  L_\tau^5 L_y^{10})([0,\infty)\times \B^3), \quad T \mapsto f^{T,\omega}
\end{equation}
is almost surely continuous. 
\end{proposition}

\begin{remark}\label{prob:rem_probabilistic_gain}
In the deterministic setting, Donninger \cite{Donninger17} relied on the \( L_\tau^2 L_y^\infty\)-Strichartz estimate to solve the nonlinear Cauchy problem. Unfortunately, the radial randomization does not lead to a probabilistic gain in the \( L_\tau^2 L_y^\infty\)-Strichartz estimate. This problem already occurs in Cartesian coordinates and is discussed in \cite[Remark 3.2]{Bringmann18}. Instead of \( L_\tau^2L_y^\infty\), we work with the \( L_\tau^1 L_y^2 \) and \( L_\tau^2 L_y^4 \)-norms, which are sufficient to close the contraction argument but lead to the weaker conclusion in Theorem \ref{main_theorem}, see Remark \ref{remark:different_norms_theorem}. 
\end{remark}

\begin{lemma}[Probabilistic Bernstein estimate]\label{prob:lem_sobolev}
Let \( f\in L_{\rad}^2(\R^3) \) and \( N \geq 1 \). Then, it holds for all \( 2 \leq p \leq \infty \) and all \( r \geq 1 \) that 
\begin{equation}\label{prob:eq_sobolev}
\| Q_N f^\omega \|_{L_\omega^r L_x^p(\Omega \times \R^3)} \lesssim_p \sqrt{r} N^{1-\frac{2}{p}+} \| Q_N f \|_{L_x^2(\R^3)}. 
\end{equation}
We also obtain that 
\begin{equation}\label{prop:eq_linf}
\| e^{\pm it |\nabla|} Q_N f^\omega  \|_{L_\omega^r L_t^\infty L_x^p(\Omega \times [0,2] \times \R^3)} \lesssim_p \sqrt{r} N^{1-\frac{2}{p}+} \| Q_N f \|_{L^2_x(\R^3)}
\end{equation}
\end{lemma}

\begin{remark}\label{remark:optimality}
Except for the endpoints \( p=2,\infty \), one can likely improve \eqref{prob:eq_sobolev} and \eqref{prop:eq_linf} through a more detailed analysis of annular Fourier multipliers (see e.g. \cite{Chanillo84}). As can be seen from the proof of Proposition \ref{prob:prop_main} below, however, this would only improve the regularity condition in \eqref{prob:eq_condition_2} and hence does not affect the main theorem. 
\end{remark}

\begin{proof}[Proof of Lemma \ref{prob:lem_sobolev}:]
For any \( n \in \N \) satisfying \( n\sim N \), we first prove the operator bound
\begin{equation}\label{prob:eq_sobolev_operator}
\| 1_{[n,n+1]}(|\nabla|) \|_{L_\rad^2(\R^3)\rightarrow L_\rad^p(\R^3)} \lesssim N^{1-\frac{2}{p}}.  
\end{equation}
The estimate for \( p=2 \) follows directly from Plancherell's theorem. We now treat the case \( p=\infty \). To this end, let \( f \in L^2_\rad(\R^3) \) and assume that \( \supp \widehat{f} \subseteq \{ \xi \colon \|\xi\|_2 \in [n,n+1] \}\). Using \eqref{pre:eq_bessel_transform}, we have that 
\begin{equation}
|f(r)|\lesssim r^{-1} \int_{n}^{n+1} |\sin(r\nu)| |\widehat{f}(\nu)| \nu \dnu \lesssim   \int_n^{n+1} |\widehat{f}(\nu)| \nu^2 \dnu \lesssim N \| \nu \widehat{f} \|_{L^2_\nu([0,\infty))} = N \| f\|_{L_x^2(\R^3)}. 
\end{equation}
The general case \( 2 \leq p \leq \infty \) then follows from Hölder's inequality. We now proceed with the proof of \eqref{prob:eq_sobolev}. For any \( 2 \leq p < \infty \) and \( r \geq p \), it follows from Minkowski's integral inequality, Khintchine's inequality, and the operator bound \eqref{prob:eq_sobolev_operator} that
\begin{align*}
&\| Q_N f^\omega \|_{L_\omega^r L_x^p} = \| \sum_{n\sim N} X_n 1_{[n,n+1]}(|\nabla|) Q_N f \|_{L_\omega^r L_x^p} 
\leq\| \sum_{n\sim N} X_n 1_{[n,n+1]}(|\nabla|) Q_N f\|_{ L_x^p L_\omega^r}  \\
&\lesssim \sqrt{r} \| 1_{[n,n+1]}(|\nabla|)Q_N f \|_{L_x^p \ell_n^2}
\lesssim \sqrt{r} \| 1_{[n,n+1]}(|\nabla|)Q_N f \|_{ \ell_n^2 L_x^p}
\lesssim \sqrt{r} N^{1-\frac{2}{p}}\| 1_{[n,n+1]}(|\nabla|) Q_N f\|_{ \ell_n^2 L_x^2} \\
&\lesssim \sqrt{r} N^{1-\frac{2}{p}} \| Q_N f \|_{L_x^2}.
\end{align*}
The restriction to \( r \geq p \) can then be removed by using Hölder's inequality in \( \omega \). If \( p=\infty \), we let \( 2\leq \widetilde{p} < \infty\) and obtain from the (deterministic) Bernstein inequality that 
\begin{equation}
\| Q_N f^\omega \|_{L^\infty(\R^3)} \lesssim N^{\frac{3}{\widetilde{p}}} \| Q_N f^\omega \|_{L^{\widetilde{p}}(\R^3)}
\end{equation}
By choosing \( \widetilde{p} \) sufficiently large, the case \( p=\infty \) in \eqref{prob:eq_sobolev} then follows from the same estimate for \( p <\infty \). Except for minor technical difficulties due to \( q=\infty\), the second estimate \eqref{prop:eq_linf} follows from the same argument. We refer to \cite[Lemma 3.7]{Bringmann18} for a detailed exposition of a similar argument. 
\end{proof}

\begin{proof}[Proof of Proposition \ref{prob:prop_main}:]
We first switch from similarity coordinates back into Cartesian coordinates. After a change of variables, we see that
\begin{equation}
\| f^{T,\omega}(\tau,y) \|_{L_\tau^q L_y^p([0,\infty)\times \B^3)} = \| (T-t)^{\frac{1}{2}-\frac{1}{q}-\frac{3}{p}} f^{\omega}(t,x) \|_{L_t^q L_x^p(\CT)} . 
\end{equation}
Assuming the scaling condition from \eqref{prob:eq_condition_1}, i.e.,  \( \frac{1}{q}+\frac{3}{p} = \frac{1}{2}\), we obtain that
\begin{equation*}
\sup_{T\in [\frac{1}{2},\frac{3}{2}]} \| f^{T,\omega}(\tau,y) \|_{L_\tau^q L_y^p([0,\infty)\times \B^3)} \leq \| f^\omega(t,x) \|_{L_t^q L_x^p([0,2]\times \R^3)}. 
\end{equation*}
The estimate then follows from \cite[Lemma 3.4 with \( \gamma=1\)]{Bringmann18}. The case \( (q,p)=(2,\infty) \) is not explicitly addressed there, but can be obtained by using Bernstein's inequality to exit \( L_x^\infty\) and using Hölder's inequality in time. \\

We now turn to the proof of the estimate under condition \eqref{prob:eq_condition_2}. After performing a Littlewood-Paley decomposition in \( N \) and losing a factor of \( N^{0+} \), we may assume that \( \fb_0 \) is frequency-localized on the dyadic scale \( N \). Let \( T_N := T - c N^{-\alpha} \), where \( \alpha>0 \) remains to be chosen. Using Hölder's inequality in the spatial variables, we have that 
\begin{equation}\label{prob:eq_proof_main_1}
\begin{aligned}
&\| (T-t)^{\frac{1}{2}-\frac{1}{q}-\frac{3}{p}} f^\omega(t,x) \|_{L_t^q L_x^p(\CT) } \\
&\leq \| 1_{[0,T_N]}(t) ~(T-t)^{\frac{1}{2}-\frac{1}{q}-\frac{3}{p}} f^\omega(t,x) \|_{L_t^q L_x^p(\CT)} 
+ \| 1_{[T_N,T]}(t) ~ (T-t)^{\frac{1}{2}-\frac{1}{q}-\frac{3}{p}} f^\omega(t,x) \|_{L_t^q L_x^p(\CT)} \\
&\leq  \| (T-t)^{\frac{1}{2}-\frac{1}{q}-\frac{3}{p}}\|_{L_t^q([0,T_N])} \cdot \| f^\omega(t,x) \|_{L_t^\infty L_x^p([0,2]\times \R^3)} \\
&+ \| (T-t)^{\frac{1}{2}-\frac{1}{q}}\|_{L_t^q([T_N,T])} \cdot \| f^\omega(t,x) \|_{L_t^\infty L_x^\infty([0,2]\times \R^3)}.
\end{aligned}
\end{equation}
Using \( p \leq 6 \), we obtain that 
\begin{equation*}
\| (T-t)^{\frac{1}{2}-\frac{1}{q}-\frac{3}{p}}\|_{L_t^q([0,T_N])} \lesssim (T-T_N)^{\frac{1}{2}-\frac{3}{p}} \lesssim N^{(\frac{3}{p}-\frac{1}{2}) \alpha}
\end{equation*}
and 
\begin{equation*}
 \| (T-t)^{\frac{1}{2}-\frac{1}{q}}\|_{L_t^q([T_N,T])} \lesssim (T-T_N)^{\frac{1}{2}} \lesssim N^{-\frac{\alpha}{2}}. 
\end{equation*}
By inserting this into \eqref{prob:eq_proof_main_1} and using Lemma \ref{prob:lem_sobolev}, it follows that 
\begin{align*}
&\big\| \sup_{T\in [\frac{1}{2},\frac{3}{2}]} \| f^{T,\omega} \|_{L_\tau^q L_y^p([0,\infty)\times \B^3)} \big \|_{L_\omega^r(\Omega)} \\
&\lesssim  N^{(\frac{3}{p}-\frac{1}{2})\alpha} \| f^\omega(t,x) \|_{L_r^\omega L_t^\infty L_x^p(\Omega \times [0,2] \times \R^3)} 
+ N^{-\frac{\alpha}{2}}\| f^\omega(t,x) \|_{L_r^\omega L_t^\infty L_x^\infty(\Omega \times [0,2] \times \R^3)} \\
&\lesssim \sqrt{r} N^{1-\frac{\alpha}{2}+} ( N^{\frac{3\alpha-2}{p}}+1)~ \| \fb_0 \|_{\mathcal{H}^0_\rad(\R^3)}. 
\end{align*}
The estimate then follows by choosing \( \alpha = 2/3 \). \\

Finally, the continuity statement \eqref{prob:eq_continuity} follows from the previous estimates and a softer argument. Indeed, for initial data with frequency support inside a single dyadic scale \(N \), the continuity follows from the continuity of 
\begin{equation}\label{prob:eq_continuity_aux}
\Big[\frac{1}{2},\frac{3}{2} \Big] \rightarrow L_t^q L_x^p( \R \times \R^3), ~ T \mapsto (T-t)^{\frac{1}{2}-\frac{1}{q}-\frac{3}{p}} 1_{\CT}(t,x). 
\end{equation}
The continuity of \eqref{prob:eq_continuity_aux} can be seen most easily through uniform bounds in \( L_t^{\widetilde{q}} L_x^{\widetilde{p}} \) for \( (\widetilde{q},\widetilde{p})=(q+,p+)\), truncating the integrands at large values, and the usual \( \epsilon/3\)-argument.  Due to the \eqref{prob:eq_main} and the strict inequality for \( s \), the sum over dyadic scales converges absolutely in \( L_\tau^q L_y^p \) for \( (q,p) \in \{ (1,2), (5,10)\} \) and hence preserves the continuity in \( T \). 
\end{proof}

\section{The nonlinear problem}\label{section:nonlinear}

In this section, we perform the two-step procedure from Section \ref{section:duhamel}. This section is similar to the nonlinear analysis in \cite{Donninger17,DS12}. 

\subsection{The modified integral equation}

For any \( \phib_0 \in \Hone \) and any spherically symmetric function \( f \colon [0,\infty) \times \B^3 \rightarrow \R \), we define the operator
\begin{equation}
\K( \phib ) = (1-P) \Big( S(\tau) \phib_0 + \int_0^\tau S(\tau-\sigma) \mathcal{N}(\phi_1,f)(\sigma) \dsigma \Big) - P \int_\tau^\infty e^{\tau-\sigma} \mathcal{N}(\phi_1,f)(\sigma) \dsigma. 
\end{equation}
We also define the norm 
\begin{equation}
\| \phib \|_{\mathcal{Z}} := \| \phib \|_{C_\tau \mathcal{H}^1_{\rad}([0,\infty)\times \B^3)} + \| \phi_1 \|_{L_\tau^2 L_y^\infty([0,\infty)\times \B^3)}.
\end{equation}
For any \( \delta >0 \), we define the corresponding \( \delta\)-ball by 
\begin{equation}
\mathcal{Z}_\delta:= \Big\{ \phi \in C_\tau \mathcal{H}^1_\rad([0,\infty)\times \B^3)\colon \| \phib \|_{\mathcal{Z}}\leq \delta \Big\}. 
\end{equation}
Using Sobolev embedding and Hölder's inequality, we also obtain that 
\begin{equation*}
\| \phi_1 \|_{L_\tau^5 L_y^{10}([0,\infty)\times \B^3)} \lesssim \| \phi_1 \|_{L_\tau^\infty L_y^6([0,\infty)\times \B^3)}^{\frac{3}{5}}\| \phi_1 \|_{L_\tau^2 L_y^\infty([0,\infty)\times \B^3)}^{\frac{2}{5}} \lesssim\| \phib \|_{\mathscr{Z}}. 
\end{equation*}

\begin{lemma}[Solution of the modified integral equation]\label{nl:lem_contraction}
Let \( 0 < \delta,c \leq 1 \) be sufficiently small. Assume that \( \phib_0 \in \Hone \) and \( f \colon [0,\infty)\times \B^3 \rightarrow \R \) satisfy
\begin{equation}\label{nl:eq_smallness_condition}
\| \phib_0 \|_{\Hone} \leq \delta \qquad \text{and} \qquad \| f \|_{L_\tau^1 L_y^2([0,\infty)\times \B^3)}, \| f \|_{L_\tau^5 L_y^{10}([0,\infty)\times \B^3)}\leq c \delta. 
\end{equation}
Then, there exists an absolute constant \( C \geq 1 \)  and a unique solution \( \phib \in \mathcal{Z}_{C\delta} \) of \( \phib = \K (\phib) \). In addition, the data-to-solution map from \( (\phib_0, f) \) satisfying \eqref{nl:eq_smallness_condition} to \( \phib \) is continuous. Furthermore, the solution satisfies the estimate
\begin{equation}\label{nl:eq_estimate_N}
\| \mathcal{N}(\phi_1,f) \|_{\LH} \lesssim c\delta + \delta^2. 
\end{equation}
\end{lemma}

\begin{proof}
We use a contraction mapping argument. We first prove for all \( \phib \in \mathcal{Z}_{C\delta} \) that 
\begin{equation}\label{nl:eq_estimate_nl}
\| \mathcal{N}(\phi_1,f) \|_{\LH} \lesssim c \delta + (C\delta)^2 + (C\delta)^5. 
\end{equation}
Indeed, we obtain from \eqref{re:eq_estimate_N} that 
\begin{align*}
&\| \mathcal{N}(\phi_1,f) \|_{\LH} \\
&\lesssim \| N(\phi_1,f) \|_{L_\tau^1 L_y^2([0,\infty)\times \B^3)} \\
&\lesssim \| f \|_{L_\tau^1 L_y^2([0,\infty)\times \B^3)} + \| f \|_{L_\tau^5 L_y^{10}([0,\infty)\times \B^3)}^5 + \| \phi_1 \|_{L_\tau^2 L_y^4([0,\infty)\times \B^3)}^2 + \| \phi_1 \|_{L_\tau^5 L_y^{10}([0,\infty)\times \B^3)}^5 \\
&\lesssim c \delta + (C\delta)^2 + (C\delta)^5. 
\end{align*}
We now estimate the contributions of the \( I-P \) and \( P \)-terms in \( \K \) separately. Using the homogeneous Strichartz estimate from Proposition \ref{re:prop_deterministic_strichartz}, we have that 
\begin{equation}
\| S(\tau) (I-P) \phib_0 \|_{\mathcal{Z}} \lesssim \| \phib_0 \|_{\Hone} \lesssim \delta. 
\end{equation}
By using the inhomogeneous Strichartz estimate from  Proposition \ref{re:prop_deterministic_strichartz}, we also have that 
\begin{equation}
\Big \| \int_0^\tau S(\tau-\sigma) (I-P) \mathcal{N}(\phi_1,f)(\sigma) \dsigma \Big \|_{\mathcal{Z}} \lesssim \| \mathcal{N}(\phi_1,f) \|_{\LH} . 
\end{equation}
Using the continuity of \( P \colon \Hone \rightarrow \Hone \), we also obtain 
\begin{equation}
\begin{aligned}
&\Big \| P \int_\tau^\infty e^{\tau-\sigma} \mathcal{N}(\phi_1,f)(\sigma) \dsigma \Big\|_{C_\tau \mathcal{H}^1_\rad([0,\infty)\times \B^3)} \\
&\lesssim \sup_{\tau\geq 0} \int_\tau^\infty e^{\tau-\sigma} \| \mathcal{N}(\phi_1,f) \|_{\Hone} \dsigma \lesssim \| \mathcal{N}(\phi_1,f)\|_{\LH}. 
\end{aligned}
\end{equation}
Finally, we obtain from Young's inequality and Lemma \ref{re:lem_semigroup} that 
\begin{equation}\label{re:eq_estimate_Lyinf}
\begin{aligned}
&\Big \| \Big(P \int_\tau^\infty e^{\tau-\sigma} \mathcal{N}(\phi_1,f)(\sigma) \dsigma \Big)_1 \Big\|_{L_\tau^2 L_y^\infty([0,\infty)\times \B^3)} \\
&\lesssim \Big \| \int_0^\infty 1_{(-\infty,0]}(\tau-\sigma) e^{\tau-\sigma}  \| \mathcal{N}(\phi_1,f)(\sigma)\|_{\Hone} \dsigma \Big\|_{L_\tau^2([0,\infty))} \cdot \| g_1 \|_{L_y^\infty(\B^3)} \\
&\lesssim \| 1_{(0,\infty]}(\tau) e^{\tau} \|_{L_\tau^2(\R)} \| \mathcal{N}(\phi_1,f) \|_{\LH} \lesssim  \| \mathcal{N}(\phi_1,f) \|_{\LH} . 
\end{aligned}
\end{equation}
By combining \eqref{nl:eq_estimate_nl}-\eqref{re:eq_estimate_Lyinf}, we obtain for all \( \phib\in \mathcal{Z}_{C\delta} \) that 
\begin{equation*}
\| \K ( \phib ) \|_{\mathcal{Z}} \lesssim \delta + c \delta + (C\delta)^2 + (C\delta)^5. 
\end{equation*}
The self-mapping property then follows by first taking \( C \geq 1 \) large enough to absorb the implicit constant and \( 0 < c,\delta\leq 1 \) sufficiently small. Once we constructed the solution, \eqref{nl:eq_estimate_nl} also implies \eqref{nl:eq_estimate_N}.\\

A standard modification of this argument shows for any  radial \( f,f^\prime \in (L_\tau^1 L_y^2 \medcap L_\tau^5 L_y^{10})([0,\infty)\times \B^3) \), \( \phib_0,\phib_0^\prime \in \Hone \) satisfying \eqref{nl:eq_smallness_condition}, and 
\( \phib,\phib^\prime \in \mathcal{Z}_{C\delta} \)  that
\begin{equation*}
\begin{aligned}
&\| \K (\phib) - K_{\phib_0^\prime, f^\prime} (\phib^\prime) \|_{\mathcal{Z}} \\
&\leq \frac{1}{2} \| \phib - \phib^\prime \|_{\mathcal{Z}} + C^\prime \Big( \| \phib_0 - \phib_0^\prime \|_{\Hone} + \| f - f^\prime \|_{L_\tau^1 L_y^2([0,\infty)\times \B^3)} +  \| f - f^\prime \|_{L_\tau^5 L_y^{10}([0,\infty)\times \B^3)} \Big). 
\end{aligned}
\end{equation*}
This implies the existence of a unique fixed point and the continuity of the data-to-solution map. 
\end{proof}

\subsection{\protect{On the blowup time $T$}}

In this subsection, we show that we can choose a blowup time \( T \) such that the modified and original Duhamel integrals coincide. 

\begin{lemma}[Choice of \( T \)]\label{nl:lem_choice_T}
Let \( 0 < \delta, c \leq 1 \) be sufficiently small. For any \( T \in [1-\delta,1+\delta] \), let \( f^T\colon [0,\infty)\times \B^3 \rightarrow \R \) be spherically symmetric. We assume the smallness condition 
\begin{equation}
\sup_{T\in [1-\delta,1+\delta]}  \big( \| f^T \|_{L_\tau^1 L_y^2([0,\infty)\times \B^3)} +  \| f^T \|_{L_\tau^5 L_y^{10}([0,\infty)\times \B^3)}  \big) \leq c\delta
\end{equation}
and the continuity of 
\begin{equation}
[1-\delta,1+\delta]\rightarrow (L_\tau^1 L_y^2 \medcap L_\tau^5 L_y^{10})([0,\infty)\times \B^3), ~ T \mapsto f^T. 
\end{equation}
Furthermore, let 
\begin{equation}
\phib_0^T := \begin{pmatrix} T^{\frac{1}{2}} u^{(1)}(0,T\cdot) - \kappa \\ T^{\frac{3}{2}} \partial_t u^{(1)}(0,T\cdot) - \frac{\kappa}{2} \end{pmatrix}.
\end{equation}
Then, there exists a unique fixed point \( \phib^T \in \mathcal{Z}_{C\delta} \) satisfying \( \phib^T = K_{\phib_0^T,f^T}(\phib^T) \). Furthermore, there exists a time \( T \in [1-\delta,1+\delta] \) such that
\begin{equation}\label{nl:eq_choice_T}
P \Big[  \phib_0^T + \int_0^\infty e^{-\sigma} \mathcal{N}(\phi_1^T,f^T)(\sigma) \dsigma \Big] =0. 
\end{equation}
\end{lemma}

\begin{proof}
We first study the variation of the initial data in \( T \). From the definition of the ODE-blowup and explicit formula for the unstable mode \( \g \) from Lemma \ref{re:lem_semigroup}, we obtain that  
\begin{equation}\label{nl:eq_variation_T_initial_data}
\phib_0^T = \kappa \begin{pmatrix} T^{\frac{1}{2}} - 1 \\  \frac{1}{2} (T^{\frac{3}{2}}-1) \end{pmatrix} \qquad \text{and} \qquad \partial_T \phib_0^T\Big|_{T=1} = \frac{\kappa}{4} \begin{pmatrix} 2\\ 3 \end{pmatrix} = \frac{\kappa}{4} \g. 
\end{equation} 
Recalling the definition of the \( \Hone \)-norm from \eqref{re:eq_hone} , we have that
\begin{equation*}
 \frac{\kappa}{4} \| \g \|_{\Hone} = \frac{\kappa}{4} \sqrt{3+4}< 1. 
\end{equation*}
This implies that \( \| \phib_0^T \|_{\Hone} \leq \delta\) for all \( T \in [1-\delta,1+\delta] \). Using Lemma \ref{nl:lem_contraction}, we then obtain the existence and uniqueness of \( \phib^T \). Since \( P \g = \g \), we also obtain from \eqref{nl:eq_variation_T_initial_data} and Lemma \ref{nl:lem_contraction} that 
\begin{equation}
P \Big[  \phib_0^T + \int_0^\infty e^{-\sigma} \mathcal{N}(\phi_1^T,f^T)(\sigma) \dsigma \Big] = \Big( \frac{\kappa}{4} (T-1) +a(T) \Big) \g, 
\end{equation}
where \( a \colon [1-\delta,1+\delta] \rightarrow \R \) is a continuous function satisfying 
\begin{equation*}
|a(T)| \lesssim (T-1)^2 +c \delta +\delta^2 \lesssim  c \delta+ \delta^2. 
\end{equation*}
If \( 0<c,\delta\leq 1 \) are sufficiently small, we see that the continuous function
\begin{equation}\label{nl:eq_function}
\big[ 1-\delta, 1+\delta \big] \rightarrow \R, ~ T \mapsto \frac{\kappa}{4} (T-1) + a(T)
\end{equation}
is negative at \( T = 1-\delta \) and positive at \( T =1+\delta \). Using the intermediate value theorem, we then see that \eqref{nl:eq_function} has a zero, which implies \eqref{nl:eq_choice_T} for some \( T \in [1-\delta,1+\delta]\). 
\end{proof}

\subsection{Proof of the main theorem}
By collecting the previous results, we now obtain a short proof of Theorem \ref{main_theorem}.

\begin{proof}[Proof of Theorem \ref{main_theorem}]
By Lemma \ref{prelim:lemma_tail} and Proposition \ref{prob:prop_main}, we have the estimate 
\begin{equation*}
\sup_{T\in[\frac{1}{2},\frac{3}{2}]} \| f^{T,\omega} \|_{L_\tau^1 L_y^2([0,\infty)\times \B^3)}, \sup_{T\in[\frac{1}{2},\frac{3}{2}]} \| f^{T,\omega} \|_{L_\tau^5 L_y^{10}([0,\infty)\times \B^3)} \leq c \delta 
\end{equation*}
and the continuity of 
\begin{equation*}
\Big[ \frac{1}{2}, \frac{3}{2}\Big] \rightarrow \big( L_\tau^1 L_y^2 \medcap L_\tau^5 L_y^{10} \big)([0,\infty)\times \B^3) , T \mapsto f^{T,\omega} 
\end{equation*}
for an event \( \Omega_{\fb_0,c,\delta} \subseteq \Omega \) with probability
\begin{equation}
\mathbb{P}( \Omega_{\fb_0,c,\delta}) \geq 1 - 2 \exp\Big( - c^\prime \frac{\delta^2}{\| \fb_0\|_{\Hs}^2}\Big).
\end{equation}
By using Lemma \ref{nl:lem_choice_T}, we obtain a time \( T \in [1-\delta,1+\delta] \) and a function \( \phib^T \in \mathcal{Z}_{C\delta} \) satisfying 
\begin{equation*}
\phib^T(\tau) = S(\tau) (1-P) \phib_0^T + \int_0^\tau S(\tau-\sigma) (1-P) \mathcal{N}(\phi_1^T,f^{T,\omega}) \dsigma - \int_\tau^\infty e^{\tau-\sigma} P \mathcal{N}(\phi_1^T,f^{T,\omega}) \dsigma. 
\end{equation*}
and 
\begin{equation*}
P \Big[ \phib_0^T + \int_0^\infty e^{-\sigma} \mathcal{N}(\phi_1^T,f^{T,\omega}) \dsigma \Big]  = 0. 
\end{equation*}
From the reformulation of the Cauchy problem in Section \ref{section:reformulation}, we see that the solution \( u \) of \eqref{rnlw} in \( \CT \) is given in similarity coordinates by \( \kappa + f^{T,\omega} + \phi_1^T \). Since \( \kappa \) corresponds to the ODE-blowup, it follows that 
\begin{align*}
&\| (T-t)^{-\frac{3}{4}} (u-u^{(T)}) \|_{L_t^2L_x^4(\CT)} + \| u - u^{(T)} \|_{L_t^5 L_x^{10}(\CT)} \\
&= \| f^{T,\omega} + \phi_1^{T} \|_{L_\tau^2 L_y^4([0,\infty)\times \B^3)} +  \| f^{T,\omega} + \phi_1^{T} \|_{L_\tau^5 L_y^{10}([0,\infty)\times \B^3)} \lesssim \delta. 
\end{align*}
This proves \eqref{intro:eq_bound} and hence the main theorem. 
\end{proof}

\renewcommand*{\bibfont}{\small}
\bibliography{Blowup}
\bibliographystyle{hplain}

\Addresses
\end{document}